\documentclass[11pt]{article}

\usepackage[a4paper,margin=2.9cm]{geometry}
\usepackage[english]{babel}
\usepackage{parskip}
\usepackage{amssymb}
\usepackage{amsmath, amsthm, amssymb}
\usepackage{enumerate}
\usepackage{enumitem}
\usepackage{graphicx}
\usepackage{bbm} 
\usepackage{mathrsfs} 
\usepackage{subcaption}
\usepackage{verbatim}
\usepackage{hyperref}
\usepackage[percent]{overpic}

\usepackage{tikz}
\usetikzlibrary{positioning}

\makeatletter
\renewcommand\section{\@startsection{section}{1}{\z@}%
                                  {-3.5ex \@plus -1ex \@minus -.2ex}%
                                  {2.3ex \@plus.2ex}%
                                  {\normalfont\Large\bfseries}}
\makeatother

\newtheoremstyle{examplestyle}
  {4mm}
  {2mm}
  {\slshape}
  {0pt}
  {\bfseries}
  {. }
  {0mm}
  {}

\theoremstyle{examplestyle}
  \newtheorem{Theorem}{Theorem}[section]
  \newtheorem{Lemma}[Theorem]{Lemma}
  
 \newtheorem*{Defs*}{Definitions}

  \newtheorem{Corollary}[Theorem]{Corollary}

\newtheorem{Conjecture}[Theorem]{Conjecture}

\title{Squared chromatic number without claws or large cliques}

\author{
Wouter Cames van Batenburg
\thanks{Department of Mathematics, Radboud University Nijmegen, Postbus 9010, 6500 GL Nijmegen, The Netherlands. \href{mailto:w.camesvanbatenburg@math.ru.nl}{\nolinkurl{w.camesvanbatenburg@math.ru.nl}}, \href{mailto:ross.kang@gmail.com}{\nolinkurl{ross.kang@gmail.com}}. Wouter Cames van Batenburg is supported by NWO grant 613.001.217. Ross J. Kang is supported by a NWO Vidi Grant, reference 639.032.614.}
\and Ross J. Kang
}

\begin{document}

\maketitle

\begin{abstract}
Let $G$ be a claw-free graph on $n$ vertices with clique number $\omega$, and consider the chromatic number $\chi(G^2)$ of the square $G^2$ of $G$.
Writing $\chi'_s(d)$ for the supremum of $\chi(L^2)$ over the line graphs $L$ of simple graphs of maximum degree at most $d$, we prove that $\chi(G^2)\le \chi'_s(\omega)$ for $\omega \in \{3,4\}$. For $\omega=3$, this implies the sharp bound $\chi(G^2) \leq 10$. For $\omega=4$, this implies $\chi(G^2)\leq 22$, which is within $2$ of the conjectured best bound.
This work is motivated by a strengthened form of a conjecture of Erd\H{o}s and Ne\v{s}et\v{r}il.
\end{abstract}

\let\thefootnote\relax

\let\thefootnote\relax


\section{Introduction}

Let $G$ be a claw-free graph, that is, a graph without the complete bipartite graph $K_{1,3}$ as an induced subgraph. We consider the square $G^2$ of $G$, formed from $G$ by the addition of edges between those pairs of vertices connected by some two-edge path in $G$. We seek to optimise the chromatic number $\chi(G^2)$ of $G^2$ with respect to the clique number $\omega(G)$. We focus on claw-free graphs $G$ having small $\omega(G)$.

The second author with de Joannis de Verclos and Pastor~\cite{JKP16+} recently conjectured the following. 
As the class of claw-free graphs is richer than the class of line graphs (cf.~e.g.~\cite{ChSe05}), this is a significant strengthening of a notorious conjecture of Erd\H{o}s and Ne\v{s}et\v{r}il (cf.~\cite{Erd88}).
\begin{Conjecture}[de Joannis de Verclos, Kang and Pastor~\cite{JKP16+}]\label{conj:JKP}
For any claw-free graph $G$, $\chi(G^2) \le \frac14(5\omega(G)^2 -2\omega(G)+1)$ if $\omega(G)$ is odd, and $\chi(G^2) \le \frac54\omega(G)^2$ otherwise.
\end{Conjecture}
If true, this would be sharp by the consideration of a suitable blow-up of a five-vertex cycle and taking $G$ to be its line graph.
The conjecture of Erd\H{o}s and Ne\v{s}et\v{r}il is the special case in Conjecture~\ref{conj:JKP} of $G$ the line graph of a (simple) graph.
To support the more general assertion and at the same time extend a notable result of Molloy and Reed~\cite{MoRe97}, it was proved in~\cite{JKP16+} that there is an absolute constant $\varepsilon >0$ such that $\chi(G^2) \le (2-\varepsilon)\omega(G)^2$ for any claw-free graph $G$. 

In this note, our primary goal is to supply additional evidence for Conjecture~\ref{conj:JKP} when $\omega(G)$ is small.
We affirm it for $\omega(G) = 3$ and come to within $2$ of the conjectured value when $\omega(G)=4$.
Note that Conjecture~\ref{conj:JKP} is trivially true when $\omega(G) \le2$.

We write $\chi'_s(\omega)$ for the supremum of $\chi(L^2)$ over the line graphs $L$ of all simple graphs of maximum degree $\omega$. Moreover, $\chi'_{s,m}(\omega)$ 
denotes the supremum of $\chi(L^2)$ over the line graphs $L$ of all multigraphs of maximum degree $\omega$.

\begin{Theorem}\label{thm:col}
Let $G$ be a claw-free graph.
\begin{enumerate}
\item\label{itm:col,3}
If $\omega(G)=3$, then $\chi(G^2) \le 10$.
\item\label{itm:col,4}
If $\omega(G)=4$, then $\chi(G^2) \le 22$;
moreover, $\chi(G^2) \le \chi'_s(4)$.
\end{enumerate}
\end{Theorem}

Note that the suitable blown-up five-vertex cycles mentioned earlier certify that Theorem~\ref{thm:col}\ref{itm:col,3} is sharp and that $\chi'_s(4) \ge 20$.
Theorem~\ref{thm:col} extends, in~\ref{itm:col,3}, a result independently of Andersen~\cite{And92} and Hor\'ak, Qing and Trotter~\cite{HQT93}, and, in~\ref{itm:col,4}, a result of Cranston~\cite{Cra06}. These earlier results proved
 the unconditional bounds of Theorem~\ref{thm:col} in the special case of $G$ the line graph $L(F)$ of some (multi)graph $F$.

It is worth contrasting the work here and in~\cite{JKP16+} with the extremal study of $\chi(G)$ in terms of $\omega(G)$ where in general the situation for claw-free $G$ is markedly different from and more complex than that for $G$ the line graph of some (multi)graph, cf.~\cite{ChSe10}.

In fact, for both $\omega(G)\in\{3,4\}$ we show that Conjecture~\ref{conj:JKP} reduces to the special case of $G$ the line graph of a simple graph.
The techniques we use for bounding $\chi(G^2)$ are purely combinatorial. They also apply when $\omega(G)>4$ (as we describe just below), but seem to be most useful when $\omega(G)$ is small.
It is natural that different methods are applicable in the small $\omega(G)$ versus large $\omega(G)$ cases, especially since this is also true of progress to date in the Erd\H{o}s--Ne\v{s}et\v{r}il conjecture.

Naturally, one could ask, for what (small) values of $\omega(G)$ does it remain true that Conjecture~\ref{conj:JKP} is ``equivalent'' to the original conjecture of Erd\H{o}s and Ne\v{s}et\v{r}il?
In light of the work in~\cite{JKP16+}, it is conceivable that structural methods such as in~\cite{ChSe05,ChSe10} will be helpful for this question.
As an extremely modest step in this direction, we have shown the following reduction for $\omega(G) \ge 5$.

\begin{Theorem}\label{thm:col,ge5}
Fix $\omega\ge5$.
Then $\chi(G^2)\le \max\{\chi'_s(\omega), 2\omega(\omega-1)-3\}$ for every claw-free graph $G$ with $\omega(G)=\omega$.
\end{Theorem}

To be transparent, let us compare this with one of the results from~\cite{JKP16+}.
\begin{Theorem}[de Joannis de Verclos, Kang and Pastor~\cite{JKP16+}]\label{thm:JKP}
Fix $\omega\ge5$. Then $\chi(G^2)\le \max\{\chi'_{s,m}(\omega),31\}$ for every claw-free graph $G$ with $\omega(G)=\omega$.
\end{Theorem}
Combined with our Theorem~\ref{thm:col} this implies that, in terms of reducing Conjecture~\ref{conj:JKP} to those $G$ which are multigraph line graphs, only the case $\omega(G)=5$ remains with margin $2$.

We remark that, for the conjecture of  Erd\H{o}s and Ne\v{s}et\v{r}il itself when $\omega(G)\in\{5,6,7\}$ there has been little progress: respectively, a trivial bound based on the maximum degree of $G^2$ yields $41$, $61$, $85$, Cranston~\cite{Cra06} speculates that $37$, $56$, $79$ are within reach, and the conjectured values are $29$, $45$, $58$.

It gives insight to notice that the claw-free graphs with clique number at most $\omega$ are precisely those graphs each of whose neighbourhoods induces a subgraph with no clique of size $\omega-1$ and no stable set of size $3$. So a good understanding of the graphs that certify small off-diagonal Ramsey numbers can be useful for this class of problems.

\paragraph{Organisation:}

In the next section and Section~\ref{sec:greedy}, we introduce some basic tools we use.
In Section~\ref{sec:3}, we treat the case $\omega(G)=3$ and prove Theorem~\ref{thm:col}\ref{itm:col,3}.
In Section~\ref{sec:4}, we treat the case $\omega(G)=4$ and prove Theorem~\ref{thm:col}\ref{itm:col,4}.
In Section~\ref{sec:5}, we briefly consider the extension of our methods to the case $\omega(G)\ge 5$ and prove Theorem~\ref{thm:col,ge5}.

\section{Notation and preliminaries}

We use standard graph theoretic notation. For instance, if $v$ is a vertex of a graph $G$, then the neighbourhood of $v$ is denoted by $N_G(v)$, and its degree by $\deg_G(v)$. 
For a subset $S$ of vertices, we denote the neighbourhood of $S$ by $N_G(S)$ and this is always assumed to be open, i.e.~$N_G(S)=\cup_{s\in S} N_G(s) \setminus S$.
We omit the subscripts if this causes no confusion.
We frequently make use of the following simple lemmas.

Recall that the Ramsey number $R(k,\ell)$ is the minimum $n$ such that in any graph on $n$ vertices there is guaranteed to be a clique of $k$ vertices or a stable set of $\ell$ vertices.
\begin{Lemma}\label{lem:ramsey}
Let $G = (V,E)$ be a claw-free graph. For any $v\in V$, the induced subgraph $G[N(v)]$ contains no clique of $\omega(G)$ vertices and no stable set of $3$ vertices. In particular, $\deg(v) < R(\omega(G),3)$.
\end{Lemma}
\begin{proof}
If not, then with $v$ there is either a clique of $\omega(G)+1$ vertices or a claw.
\end{proof}

\begin{Lemma}\label{lem:2ndclique}
Let $G = (V,E)$ be a claw-free graph. For any $v,w\in V$ and $vw\in E$, any two distinct $x,y \in N(w)\setminus(\{v\}\cup N(v))$ are adjacent. In particular, $|N(w)\setminus(\{v\}\cup N(v))| \le \omega(G)-1$.
\end{Lemma}
\begin{proof}
If not, then $v$, $w$, $x$, $y$ form a claw. So $\{w\} \cup N(w)\setminus(\{v\}\cup N(v))$ is a clique.
\end{proof}

It is not required next that $x,y\in N(v)$, but it is the typical context in which it is used.

\begin{Lemma}\label{lem:common}
Let $G = (V,E)$ be a claw-free graph. For any $v\in V$ and $w\in N(v)$, if $N(v)\cap N(w)$ contains two non-adjacent vertices $x$ and $y$, then for any $z\in N(w)\setminus(\{v\}\cup N(v))$, either $xz\in E$ or $yz\in E$.
\end{Lemma}
\begin{proof}
If not, then $w$, $x$, $y$, $z$ form a claw.
\end{proof}

\section{A greedy procedure}\label{sec:greedy}

In this section, we describe a general inductive procedure to use vertices of small square degree to colour squares in a class of graphs. This slightly refines a procedure in~\cite{JKP16+} so that it is suitable for our specific purposes.

\begin{Lemma}\label{lem:greedy}
Let $K$ be a non-negative integer.
Suppose ${\mathcal C}_1$ and ${\mathcal C}_2$ are graph classes such that ${\mathcal C}_1$ is non-empty and closed under vertex deletion and every graph $G\in {\mathcal C}_2$ satisfies $\chi(G^2) \le K+1$. Furthermore, suppose there exists $K'\le K$ such that every graph $G\in {\mathcal C}_1$ satisfies one of the following:
\begin{enumerate}
\item $G$ belongs to ${\mathcal C}_2$;
\item\label{itm:ii} there is a vertex $v \in V(G)$ such that $\deg_{G^2}(v) \le K'$, there is a vertex $x^*\in N_G(v)$ with $\deg_{G^2}(x^*) \leq K'+1$ and the set of all vertices $x\in N_G(v)$ with $\deg_{G^2}(x) > K'+2$ induces a clique in $(G\setminus v)^2$; or
\item\label{itm:iii} there is a vertex $v \in V(G)$ such that $\deg_{G^2}(v) \le K'$ and the set of all vertices $x\in N_G(v)$ with $\deg_{G^2}(x) > K'+1$  
induces a clique in $(G\setminus v)^2$.

\end{enumerate}
For any $G\in {\mathcal C}_1$, $\chi(G^2) \le K+1$.
\end{Lemma}

\begin{proof}
We proceed by induction on the number of vertices. Since $K$ is non-negative and the singleton graph is in ${\mathcal C}_1$, the base case of the induction holds. Let $G$ be a graph in ${\mathcal C}_1$ with at least two vertices and suppose that the claim holds for any graph of ${\mathcal C}_1$ with fewer vertices than $G$ has. If $G\in {\mathcal C}_2$, then we are done by the assumption on ${\mathcal C}_2$. So it only remains to consider the second and third possibility.

We now prove the bound under assumption of case~\ref{itm:ii}. Let $v$ be the vertex guaranteed in this case and write $B$ for the set of vertices $x\in N_G(v)$ with $\deg_{G^2}(x) > K'+2$ and $S=N(v)\setminus B$. Since ${\mathcal C}_1$ is closed under vertex deletion, by induction there is a proper colouring $\varphi$ of $(G\setminus v)^2$ with at most $K+1$ colours. Since $B$ is a clique, all elements in $B$ are assigned different colours  under $\varphi$. From $\varphi$, we will now obtain a new proper $(K+1)-$colouring $\varphi'$ of $(G\backslash v)^2$ such that \textit{all} elements of $N_G(v)$ have different colours. 

First we uncolour all vertices in $S$. We then wish to recolour them with pairwise distinct colours as follows. Given $s\in S$, we say a colour in $\left\{1,\ldots, K+1 \right\}$ is \textit{available} to $s$ if it is distinct from any colour assigned by $\varphi$ to the vertices in $N_{G^2}(s)\setminus(\{v\}\cup S)$. Since $\deg_{G^2}(s) \le K'+2\le K+2$ and $\{v\}\cup S\setminus\{s\} \subseteq N_{G^2}(s)$, the number of colours available to $s$ is at least $K+1- (\deg_{G^2}(s) -|\left\{v\right\} \cup S\setminus\{s\}|) \leq K+1 - ((K+2)- |S|)=|S|-1$. Furthermore, since $x^*\in S$ and $\deg_{G^2}(x^*) \leq K'+1\leq K+1$, the number of colours available to $x^*$ is at least $|S|$.
 Since the complete graph on $|S|$ vertices is (greedily) list colourable for any list assignment with $|S|-1$ lists of size $|S|-1$ and one list of size $|S|$, it follows that we can recolour the vertices of $S$ with pairwise distinct available colours. 

This new colouring $\varphi'$ is a proper $(K+1)$-colouring of $(G\setminus v)^2$ such that all elements in $N_G(v)$ have different colours. Since $\deg_{G^2}(v) \le K' \le K$, there is at least one colour not appearing in $N_{G^2}(v)$ that we can assign to $v$ so that together with $\varphi'$ we obtain a proper $(K+1)$-colouring of $G^2$. 

The proof under assumption~\ref{itm:iii} is nearly the same as under~\ref{itm:ii}. Defining $B$ for the set of vertices $x\in N_G(v)$ with $\deg_{G^2}(x) > K'+1$ and $S=N(v)\backslash B$, we obtain that every $s\in S$ has $|S|$ available colours. This allows us to complete the colouring as before.
\end{proof}

\section{Clique number three}\label{sec:3}

In this section, we prove Theorem~\ref{thm:col}\ref{itm:col,3}.
We actually prove the following result.

\begin{Theorem}\label{thm:3}
Let $G = (V,E)$ be a connected claw-free graph with $\omega(G)=3$. Then one of the following is true:
\begin{enumerate}
\item\label{itm:3,ico}
$G$ is the icosahedron;
\item\label{itm:3,line}
$G$ is the line graph $L(F)$ of a $3$-regular graph $F$; or

\item\label{itm:3,deg}
there exists $v\in V$ with $\deg_{G^2}(v) \le 9$ such that $\deg_{G^2}(x) \le 11$ for all $x\in N_G(v)$. Furthermore, either there exists $x^*\in N_G(v)$ with $\deg_{G^2}(x^*)\leq 10$, or $N_G(v)$ induces a clique in $(G\backslash v)^2$.
\end{enumerate}
\end{Theorem}

Let us first see how this easily implies Theorem~\ref{thm:col}\ref{itm:col,3}.

\begin{proof}[Proof of Theorem~\ref{thm:col}\ref{itm:col,3}]
Let ${\mathcal C}_1$ be the class of claw-free graphs $G$ with $\omega(G)\le 3$. Clearly ${\mathcal C}_1$ is non-empty and closed under vertex deletion.

Let ${\mathcal C}_2$ be the class of graphs formed by taking all claw-free graphs $G$ with $\omega(G)\le 2$, the icosahedron, and the line graphs $L(F)$ of all $3$-regular graphs $F$. If $G$ is a claw-free graph with $\omega(G)\le 2$, then $\chi(G^2)\le 5$. If $G$ is the icosahedron, then $\chi(G^2) \le 6$ is certified by giving every pair of antipodal points the same colour. If $G$ is the line graph of a $3$-regular graph, then $\chi(G^2)\le 10$ by the strong edge-colouring result due, independently, to Andersen~\cite{And92} and to Hor\'ak, Qing and Trotter~\cite{HQT93}.

Theorem~\ref{thm:3} certifies that we can apply Lemma~\ref{lem:greedy} with $K=K'=9$.
\end{proof}

\begin{proof}[Proof of Theorem~\ref{thm:3}]
First we show that either case~\ref{itm:3,ico} or~\ref{itm:3,line} applies, or that there exists a vertex $v\in V$ with $\deg_{G^2}(v) \leq 9$. At the end, we show that, for all such $v$, it also holds that $\deg_{G^2}(x)\leq 11$ for all $x\in N_G(v)$ and that furthermore these vertices either induce a clique in $(G\backslash v)^2$, or contain a vertex $x^*$ with $\deg_{G^2}(x^*)$ at most $10$.

First note that the maximum degree $\Delta(G)$ of $G$ is at most $5$. This follows from Lemma~\ref{lem:ramsey} and the fact that $R(3,3)=6$.
Moreover, note that, for any $v\in V$ with $\deg(v) = 5$, $G[N(v)]$ must be a $5$-cycle by Lemma~\ref{lem:ramsey}.

For $v\in V$ with $\deg(v) \le 2$, we have $\deg_{G^2}(v) \le 2 + 2\cdot 2 = 6$ by Lemma~\ref{lem:2ndclique}.
For $v\in V$ with $\deg(v)=3$, we have $\deg_{G^2}(v) \le 3 + 3\cdot 2 = 9$ by Lemma~\ref{lem:2ndclique}.
So in terms of proving the existence of a vertex $v$ with $\deg_{G^2}(v)\leq 9$, we can assume hereafter that the minimum degree of $G$ satisfies $\delta(G)\ge 4$. 

For $v\in V$ with $\deg(v)= 4$, we call $v$ {\em good} if the subgraph $G[N(v)]$ induced by $N(v)$ is not the disjoint union of two edges.
Assume for the moment that $G$ contains no good vertex.

If $\delta(G) = \Delta(G) = 4$, then every neighbourhood induces the disjoint union of two cliques (each of exactly two vertices).
Recall that a graph is the line graph of a graph if its edges can be partitioned into maximal cliques so that no vertex belongs to more than two such cliques and additionally, no two vertices are both in the same two cliques.
We can designate the maximal cliques as follows: for $v\in V$ and a clique $C$ that is maximal in $N(v)$, designate $v\cup C$ as a maximal clique for the requisite edge partition.
Indeed, every edge $v_1v_2$ is designated as part of one of the cliques, either from the perspective of $v_1$ or of $v_2$. Moreover, the clique to which $v_1v_2$ is designated does not differ depending on the endpoint from which the perspective is taken, since every neighbourhood induces the disjoint union of two cliques.
As each of the designated cliques has exactly three vertices, it follows that $G$ is the line graph $L(F)$ of a $3$-regular graph $F$.

If, on the other hand, there exists $v\in V$ with $\deg(v)=5$, then consider $x\in N(v)$. Since $G[N(v)]$ is a $5$-cycle, $x$ has three neighbours $y_1,v,y_3$ that induce a $3$-vertex path $y_1vy_3$. This means $G[N(x)]$ is not the union of two cliques. By our assumption that no vertex is good, it follows that $x$ has degree $5$. 
So $G$ is the icosahedron, the unique connected graph in which every neighbourhood induces a $5$-cycle. (Uniqueness can be easily seen by constructing the graph up to distance $2$ from $v$ in the only possible way respecting induced $5$-cycles, and then noting that the vertices at distance $2$ from $v$ induce a $5$-cycle and that they all need to be adjacent to a $12$th and final vertex.)

From now on, let $v\in V$ be a good vertex.
We next show that $|N(N(v))\setminus\{v\}| \le 5$ (which implies $\deg_{G^2}(v) \le 9$).

Since $G[N(v)]$ has no stable set of three vertices and $v$ is good, $G[N(v)]$ has at least three edges. Moreover, since $G[N(v)]$ has no clique of three vertices, we can write $N(v) = \{x_1,x_2,x_3,x_4\}$ such that $x_1x_2,x_2x_3,x_3x_4\in E$ and $x_1x_3,x_2x_4\notin E$.  By Lemma~\ref{lem:2ndclique}, both $x_1$ and $x_4$ have at most $2$ neighbours outside $\left\{v\right\} \cup N(v)$. So it suffices to show that $\left\{x_2,x_3\right\}$ cannot have two neighbours outside $\left\{v\right\} \cup N(v)$ which are not neighbours of $\left\{x_1,x_4\right\}$. By contradiction, let $p,q$ be these vertices. Without loss of generality, $p$ is a neighbour of $x_2$. Then $p$ is ajdacent to $x_3$, for otherwise $x_1x_2x_3p$ would be a claw. Similarly, $q$ is adjacent to both $x_2$ and $x_3$. But then $pq$ is an edge (otherwise $x_1pqx_2$ would be a claw), so that $x_2x_3pq$ is a $K_4$. Contradiction. This concludes the proof that there exists a vertex $v$ with $\deg_{G^2}(v) \leq 9$.

From now on, let $v$ be one of the vertices for which we showed above that $\deg_{G^2}(v) \le 9$. In particular, if $v$ has degree $4$ then it is a good vertex.

Let us call a vertex $x$ {\em extremely bad} if $\deg_{G^2}(x) \ge 12$. We already observed that no vertex $x$ with $\deg(x)\le 3$ is extremely bad. If $\deg(x)=5$, then $N(x)$ induces a $5$-cycle and so by Lemma~\ref{lem:common} every vertex in $N(N(x))\setminus\{x\}$ has at least two neighbours in $N(x)$, so $|N(N(x))\setminus\{x\}|\le 5$. So a vertex $x$ can only be very bad if $\deg(x)=4$ and it is not good. In particular, by Lemma~\ref{lem:2ndclique}, not only does the neighbourhood of $x$ induce a disjoint union of two edges, but also the same is true for every neighbour of $x$. This implies that $N(v)$ does not contain an extremely bad vertex.

Let us call a vertex $x$ {\em very bad} if $\deg_{G^2}(x)=11$. We are done if there exists $x^* \in N_G(v)$ with $\deg_{G^2}(x^*)\leq 10$. So we may assume from now on that \textit{all} vertices in $N_G(v)$ are very bad, and we need to show that they induce a clique in $(G\setminus v)^2$. Assume for a contradiction that they do not. Since the neighbourhood of a degree $5$ vertex induces a $5$-cycle, of which the square is a clique, we may assume that $ \deg(v)\leq 4$. If $\deg(v)= 3$, then there are $x_1,x_2,x_3 \in N(v)$ such that $x_1x_2,x_2x_3 \notin E(G)$, so $\deg_G(x_2)\leq 3$, so $\deg_{G^2}(x_2)\leq 9$, contradicting that $x_2$ is very bad. Similarly if $\deg(v)\leq 2$. Thus we have reduced to the case that $v$ is a good vertex (of degree $4$). As argued before, we can then write $N(v)= \left\{x_1,x_2,x_3,x_4 \right\}$ such that $x_1x_2, x_2x_3, x_3x_4 \in E$ and $x_1x_3,x_2x_4 \notin E$. Since $N(v)$ does not induce a clique in $(G\setminus v)^2$, it follows that also $x_1x_4\notin(E)$. Therefore $\deg_G^2(x_1) \leq 10$, contradicting that $x_1$ is very bad. This completes the proof.
\end{proof}

\section{Clique number four}\label{sec:4}

The proof of Theorem~\ref{thm:3} suggests the following rougher but more general phenomenon. 
This follows from Lemmas~\ref{lem:2ndclique} and~\ref{lem:common} together with a double-counting argument.

For $G=(V,E)$ and $v\in V$, we define the following subset of $N(v)$:
\begin{align*}
Z(v) := \{w\in N(v) \mid \exists x, y\in N(v)\text{ such that }xw,wy\in E\text{ and }xy\notin E \}.
\end{align*}

\begin{Lemma}\label{lem:Z}
Let $G = (V,E)$ be a claw-free graph. For any $v\in V$,
\begin{align*}
|N(N(v))\setminus\{v\}|
& \le \sum_{w\in N(v)\setminus Z(v)} |N(w)\setminus(\{v\}\cup N(v))| + \frac12\sum_{w\in Z(v)} |N(w)\setminus(\{v\}\cup N(v))|\\
& \le \left(\deg(v)-\frac12|Z(v)|\right)(\omega(G)-1).
\end{align*}
\end{Lemma}
\begin{proof}

Let $w \in Z(v)$. By Lemma~\ref{lem:common}, any $x \in N(w) \setminus ({\left\{v\right\} \cup N(v)})$  also satisfies $x \in N(y) \setminus ({\left\{v\right\} \cup N(v)})$ for some $y\in N(v) \setminus \left\{w\right\}$. So 
\[
|N(N(v)) \setminus \left\{v\right\}| = \sum_{w \in N(v)} \sum_{x \in N(w)\setminus (\left\{v\right\} \cup N(v))}  \frac{1}{|\{  u \in N(v) \mid x \in N(u)\}|}
\]
is at most $\sum_{w\in N(v)\setminus Z(v)} |N(w)\setminus(\{v\}\cup N(v))| + \frac12\sum_{w\in Z(v)} |N(w)\setminus(\{v\}\cup N(v))|$.  Now apply Lemma~\ref{lem:2ndclique}.
\end{proof}

This has the following immediate consequence.

\begin{Corollary}\label{cor:Z}
Let $G = (V,E)$ be a claw-free graph. For any $v\in V$ with $\deg(v)\ge 2\omega(G)-1$, we have $Z(v)=N(v)$ and therefore
\begin{align*}
|N(N(v))\setminus\{v\}|
& \le \frac12\sum_{w\in N(v)} |N(w)\setminus(\{v\}\cup N(v))|
 \le \frac12\deg(v)(\omega(G)-1).
\end{align*}
\end{Corollary}

\begin{proof}
Let $w\in N(v)$ and consider $N_{G[N(v)]}(w)$.
By Lemma~\ref{lem:2ndclique}, 
 $\deg_{G[N(v)]}(w) \ge \deg(v)-(\omega(G)-1) -1 \ge \omega(G)-1$.
Then $N_{G[N(v)]}(w)$ contains a pair of non-adjacent vertices, or else $\{v,w\}\cup N_{G[N(v)]}(w)$ is a clique of $\omega(G)+1$ vertices. As $w$ was arbitrary, we have just shown that $Z(v) = N(v)$. So the result follows from Lemma~\ref{lem:Z}. 
\end{proof}

We now prove the following result. Similarly to what we saw if $\omega(G)=3$, this implies for any claw-free $G$ with $\omega(G)=4$ that $\chi(G^2)\le 22$ by Lemma~\ref{lem:greedy} with $K=21$ and $K'=19$, due to a result of Cranston~\cite{Cra06}. Furthermore, since $\chi'_s(4) \ge 20$, we may make the choice $K=\chi'_s(4)-1$ and $K'=19$ to obtain Theorem~\ref{thm:col}\ref{itm:col,4}, i.e.~that Conjecture~\ref{conj:JKP} for $\omega(G)=4$ reduces to the corresponding case of the Erd\H{o}s--Ne\v{s}et\v{r}il conjecture.

\begin{Theorem}\label{thm:4}
Let $G = (V,E)$ be a connected claw-free graph with $\omega(G)=4$. Then one of the following is true:
\begin{enumerate}
\item\label{itm:4,line}
$G$ is the line graph $L(F)$ of a graph $F$ of maximum degree $4$; or
\item\label{itm:4,deg}
there exists $v\in V$ with $\deg_{G^2}(v) \le 19$ such that the set of all vertices $x\in N_G(v)$ with $\deg_{G^2}(x) \ge 21$ induces a clique in $(G\setminus v)^2$.
\end{enumerate}
\end{Theorem}

\begin{proof}
First we show that either case~\ref{itm:4,line} applies or that there exists a vertex $v\in V$ with $\deg_{G^2}(v) \leq 19$. At the end, we show that, for all such $v$, it also holds that the set of vertices $x\in N_G(v)$ with $\deg_{G^2}(x)\geq 21$ induces a clique in $(G\setminus v)^2$.

First note that the maximum degree $\Delta(G)$ of $G$ is at most $8$. This follows from Lemma~\ref{lem:ramsey} and the fact that $R(4,3)=9$.

For $v\in V$ with $\deg(v) \le 4$, we have $\deg_{G^2}(v) \le 4 + 4\cdot 3 = 16$ by Lemma~\ref{lem:2ndclique}.

Note that, for $v\in V$ with $\deg(v)=5$, we have $\deg_{G^2}(v) \le 5 + 5\cdot 3 = 20$ by Lemma~\ref{lem:Z}, but equality cannot occur here unless $Z(v) = \emptyset$.  (Indeed, if $Z(v)\neq \emptyset$, then

For $v\in V$ with $\deg(v)=5$ and $Z(v)=\emptyset$, $G[N(v)]$ is the disjoint union of cliques, and in particular it must be the disjoint union of an edge and a triangle.

For $v\in V$ with $\deg(v) = 7$, we have $\deg_{G^2}(v) \le 7+21/2 =17.5$ by Corollary~\ref{cor:Z}.

Let $v\in V$ with $\deg(v) = 8$. By Corollary~\ref{cor:Z}, $Z(v)=N(v)$ and so we already have $\deg_{G^2}(v) \le 8+24/2 =20$, but we want one better.
Let $w\in N(v)$. By Lemma~\ref{lem:2ndclique}, $N(v)\setminus (N_{G[N(v)]}(w) \cup \left\{w\right\})$ is a clique, so $\deg_{G[N(v)]}(w) \ge \deg(v)-\omega(G)=4$. Now $N_{G[N(v)]}(w)$ contains no clique or stable set of three vertices, or else $G$ contains a clique of $5$ vertices or a claw. We can therefore find four vertices $x_1,x_2,x_3,x_4\in N_{G[N(v)]}(w)$ such that $x_1x_2,x_3x_4\notin E$.
(There is at least one non-edge among $x_1,x_2,x_3$, say, $x_1x_2$. Since $G$ is claw-free at least one of $x_1x_3$ and $x_2x_3$ is an edge, say, $x_2x_3$. Among $x_2,x_3,x_4$, there is at least one non-edge, which together with $x_1x_2$ or $x_1x_3$ forms a two-edge matching in the complement, which is what we wanted, after relabelling.)
By Lemma~\ref{lem:common}, for every $y\in N(w)\setminus(\{v\}\cup N(v))$, either $x_1y\in E$ or $x_2y\in E$ {\em and} $x_3y\in E$ or $x_4y\in E$.
We have just shown that every vertex in $N(N(v))\setminus\{v\}$ has at least three neighbours in $N(v)$.  Therefore, $|N(N(v))\setminus\{v\}| \le \frac13\deg(v)(\omega(G)-1) = 8$ and $\deg_{G^2}(v) \le 16$.

Let $v\in V$ with $\deg(v) = 6$. By Lemma~\ref{lem:2ndclique}, the minimum degree of $G[N(v)]$ satisfies $\delta(G[N(v)]) \ge \deg(v)-\omega(G)=2$. Since $G$ contains no clique of $5$ vertices, every vertex with degree at least $3$ in $G[N(v)]$ must also be in $Z(v)$. So we know there are at most two such vertices, or else by Lemma~\ref{lem:Z} $\deg_{G^2}(v) \le 6 + \lfloor(6-3/2)\cdot 3\rfloor = 19$.
First suppose there is a vertex $w$ with degree $5$ in $G[N(v)]$. Since $N_{G[N(v)]}(w)$ contains no clique or stable set of three vertices, it must be that $G[N(v)]$ consists of $w$ adjacent to all vertices of a $5$-cycle, in which case all six vertices have degree at least $3$ in $G[N(v)]$. This contradicts that at most two vertices of degree at least $3$ are allowed in $G[N(v)]$.
Next suppose that there is a vertex $w$ with degree $4$ in $G[N(v)]$. Then there exists $w'\in N(v)$ with $ww'\notin E$. As we argued in the last paragraph, there exist $x_1,x_2,x_3,x_4\in N_{G[N(v)]}(w)$ such that $x_1x_2,x_3x_4\notin E$. Since $G$ is claw-free, it must be that $w'$ is adjacent to one of $x_1$ and $x_2$ and also to one of $x_3$ and $x_4$; without loss of generality suppose $x_1w',x_3w'\in E$. It follows that $x_1,x_3,w$ are three vertices with degree at least $3$ in $G[N(v)]$, which was not allowed. So now we have reduced to the case where $2\le \delta(G[N(v)]) \le \Delta(G[N(v)]) \le 3$ and there are at most two vertices with degree $3$ in $G[N(v)]$. Since $G$ is claw-free, there are only two possibilities for the structure of $G[N(v)]$: either it is a disjoint union of two triangles, or it is that graph with the inclusion of exactly one additional edge.

We call a vertex $v$ {\em good} if its neighbourhood structure does not satisfy one of the following:
\begin{itemize}
\item $G[N(v)]$ is the disjoint union of a singleton and a triangle;
\item $G[N(v)]$ is the disjoint union of an edge and a triangle;
\item $G[N(v)]$ is the disjoint union of an edge and a triangle plus one more edge;
\item $G[N(v)]$ is the disjoint union of two triangles;
\item $G[N(v)]$ is the disjoint union of two triangles plus one more edge; or
\item $G[N(v)]$ is the disjoint union of two triangles plus two more non-incident edges.
\end{itemize}
Recall that a graph is the line graph of a graph if its edges can be partitioned into maximal cliques so that no vertex belongs to more than two such cliques and additionally, no two vertices are both in the same two cliques.
If no vertex $v\in V$ is good, then we can designate the maximal cliques as follows: for each $v\in V$ and for any $C$ one of the two maximum cliques of $G[N(v)]$ specified in one of the cases above (this is well-defined), we designate $v\cup C$ as a maximal clique for the requisite edge partition.
Indeed, every edge $v_1v_2$ is designated as part of one of the cliques, either from the perspective of $v_1$ or of $v_2$. Moreover, the clique to which $v_1v_2$ is designated does not differ depending on the endpoint from which the perspective is taken, by a brief consideration of the six impermissible neighbourhood structures defining a good vertex.
As each of the designated cliques has at most four vertices, it follows that in this case $G$ is the line graph $L(F)$ of a graph $F$ of maximum degree $4$.

Our case analysis has shown that either no vertex of $G$ is good, in which case $G$ is the line graph of a graph of maximum degree $4$, or there is some good $v\in V$ with $\deg_{G^2}(v) \le 19$. From now on, we fix one such good vertex $v$.

Let us call a vertex $x$ {\em very bad} if $\deg_{G^2}(x) \ge 21$. We already observed that $x$ must then have $\deg(x)=6$. 
By the case analysis above, the neighbourhood of $x$ either induces a disjoint union of two triangles or is that graph plus one more edge. However, the latter case is excluded, as we will now demonstrate. Suppose the neighbourhood of a vertex $x$ induces two triangles $w_1w_2w_3$ and $w_4w_5w_6$ plus one more edge $w_1w_4$. Our goal is to derive then that $\deg_{G^2}(x) \leq 20$, so that $x$ cannot be very bad. By Lemma~\ref{lem:2ndclique}, $w_i$ has at most three neighbours outside $\left\{v\right\}\cup N(v)$, for all $i\in \left\{2,3,5,6 \right\}$. So it suffices to show that $\left\{w_1,w_4 \right\}$ cannot have three neighbours outside $\left\{v\right\} \cup N(v)$ which are not a neighbour of $\left\{w_2,w_3,w_5,w_6 \right\}$. By contradiction, let $p$, $q$, $r$ be these neighbours. Without loss of generality, $p$ is a neighbour of $w_1$. Then $p$ is also adjacent to $w_4$ (otherwise claw). The same argument applies to $q$ and $r$, so that $\left\{p,q,r \right\}$ must be complete to $\left\{w_1,w_4 \right\}$. Furthermore, by claw-freeness, $pqr$ must be a triangle. But then $\{w_1,w_4,p,q,r\}$ induces a $K_5$. Contradiction. This completes the proof that the neighbourhood of a very bad vertex induces the disjoint union of two triangles.

Let $x_1$ be a very bad vertex in $N(v)$. Since $N(x_1)$ induces two disjoint triangles (one containing $v$) it follows that $x_1$ is part of a triangle $x_1x_2x_3$ in $N(v)$ and there is no edge between $x_1$ and $N(v) \setminus \left\{x_1,x_2,x_3\right\}$. Thus each vertex in $N(v)\setminus \left\{x_1,x_2,x_3\right\}$ is at distance exactly $2$ from $x_1$ (with respect to $G$) so that $N(v)\setminus \left\{x_1,x_2,x_3\right\}$ is a clique by Lemma \ref{lem:2ndclique}. 

Suppose now that the very bad vertices in $N(v)$ do not form a clique in $(G\setminus v)^2$. Writing $N(v):=\{x_1,\ldots, x_6\}$, then there exist two very bad vertices $x_1, x_6$, say, that are at distance greater than $2$ in $G\setminus v$. By the previous paragraph, $N(v)$ is covered by two disjoint triangles. Because $v$ is good, it follows (up to symmetry of $x_1$ and $x_6$) that the following is a subgraph of the graph induced by $N(v)$: two disjoint triangles $x_1x_2x_3$ and $x_4x_5x_6$ plus two edges $x_2x_4, x_3x_4$. Note that $x_2,x_3$ and $v$ are neighbours of $x_1$ that have a common neighbour at distance $2$ from $x_1$, namely $x_4$, and separately from that, $x_2$ and $x_3$ have a common neighbour in $N(N(x_1)) \cap N(N(v))\setminus \left\{v, x_1\right\}$. It follows that $\deg_{G^2}(x_1)\leq 20$, contradicting that $x_1$ is very bad. We have shown that the very bad vertices in $N(v)$ form a clique in $(G\setminus v)^2$ and this concludes the proof.
\end{proof}

\section{Clique number at least five}\label{sec:5}

The proof of Theorem~\ref{thm:4} suggests the following refinement of Lemma~\ref{lem:Z}. This could be useful towards reductions to the line graph setting for $\omega(G)\ge 5$.

For $G=(V,E)$ and $v\in V$ and $w\in N(v)$, we define $q(w)$ to be the matching number of the complement of $G[N_{G[N(v)]}(w)]$. Note that $q(w)\ge1$ if and only if $w\in Z(v)$.

\begin{Lemma}\label{lem:Q}
Let $G = (V,E)$ be a claw-free graph. For any $v\in V$,
\begin{align*}
|N(N(v))\setminus\{v\}|
& \le \sum_{w\in N(v)} \frac{|N(w)\setminus(\{v\}\cup N(v))|}{q(w)+1}
 \le (\omega(G)-1)\sum_{w\in N(v)}\frac1{q(w)+1}.
\end{align*}
\end{Lemma}
\begin{proof}
Let $a_1b_1, a_2b_2,\ldots, a_{q(w)}b_{q(w)}$ be edges of a maximum matching in the complement of $G[N_{G[N(v)]}(w)]$. Note that $w$ and $a_1,b_1, \ldots, a_{q(w)}, b_{q(w)}$ are all distinct vertices in $N(v)\cap N(w)$. Let $x\in N(w)\setminus \left\{v\right\}$. For all $i \in \left\{ 1,\ldots, q(w) \right\}$, it holds that $wa_i, wb_i \in E$ and $a_ib_i \notin E$, so by Lemma~\ref{lem:common} $x$ is not only a neighbour of $w$, but also a neighbour of $a_i$ or $b_i$. This implies that $|\{ u \in N(v) \mid x \in N(u) \}| \geq q(w)+1$.
So
\[|N(N(v)) \setminus \left\{v\right\}| = \sum_{w \in N(v)} \sum_{x \in N(w)\setminus (\left\{v\right\} \cup N(v))}  \frac{1}{|\{  u \in N(v) \mid x \in N(u)\}|}
\]
is at most $\sum_{w\in N(v)} |N(w)\setminus(\{v\}\cup N(v))|/(q(w)+1)$. Now apply Lemma~\ref{lem:2ndclique}.
\end{proof}

Lemma~\ref{lem:Q} yields the following corollary.

\begin{Corollary}\label{cor:Q}
Let $G = (V,E)$ be a claw-free graph with $\omega(G)\ge 4$. For any $v\in V$ with $\deg(v)\ge 2\omega(G)-1$,
\begin{align*}
|N(N(v))\setminus\{v\}|
& \le \frac{\deg(v)(\omega(G)-1)}{\lceil(\deg(v)+1)/2\rceil+2-\omega(G)}.
\end{align*}
\end{Corollary}

\begin{proof}
Let $w\in N(v)$.
It suffices to establish a suitable lower bound for $q(w)$. By Lemma~\ref{lem:2ndclique}, $\deg_{G[N(v)]}(w) \ge \deg(v)-\omega(G)\ge \omega(G)-1$, and so in any subset of $N_{G[N(v)]}(w)$ with at least $\omega(G)-1$ vertices there must be at least one non-edge (or else $G$ has a clique of $\omega(G)+1$ vertices). So we can iteratively extract two vertices from $N_{G[N(v)]}(w)$ that form an edge of the complement of $G[N_{G[N(v)]}(w)]$ until at most $\omega(G)-2$ vertices remain. It follows that
\begin{align*}
q(w)
& \ge \left\lceil\frac12(\deg_{G[N(v)]}(w)-(\omega(G)-2))\right\rceil
\ge \left\lceil\frac12(\deg(v)-\omega(G)-(\omega(G)-2))\right\rceil \\
& = \lceil\deg(v)/2\rceil +1-\omega(G).
\end{align*}
If $\deg(v)$ is even, then after we have extracted $\lceil\deg(v)/2\rceil-\omega(G)$ pairs as above at least $\omega(G)$ vertices remain, call them $x_1,\dots,x_{\omega(G)}$. Among $x_1,\dots,x_{\omega(G)-1}$ there is at least one non-edge, say, $x_1x_2\notin E$ without loss of generality.

Since $\omega(G)\ge4$, there is at least one non-edge $ab$ among $x_2,\ldots,x_{\omega(G)}$, and at least one non-edge $cd$ among $x_1,x_3,\ldots x_{\omega(G)}$. The non-edges $x_1x_2, ab$ and $cd$ may not form a stable set of size three, since otherwise there would be a claw. Therefore at least two of them comprise a two-edge matching in the complement of $G[\{x_1,\dots,x_{\omega(G)}\}]$.
So indeed we have for any parity of $\deg(v)$ that
\begin{align*}
q(w)
& \ge \lceil(\deg(v)+1)/2\rceil +1-\omega(G).
\end{align*}
As $w$ was arbitrary, the result now follows from Lemma~\ref{lem:Q}.
\end{proof}

Let us now make explicit some general consequence of Corollary~\ref{cor:Q}.
An awkward but routine optimisation checks that for $k\geq 5$ and $x\in\{2k-1,2k,\dots\}$, the expression $f(x):=x+\frac{x(k-1)}{\lceil (x+1)/2\rceil+2-k}$ is maximised with $x=2k-1$ or with $x \in \left\{y,y+1\right\}$ for $y$ as large as possible. (This follows e.g.~from the facts that $f(2k-1) > f(2k) $ and that there is some $x_0> 2k-1$ such that the derivative of $f^*(x):=x+\frac{x(k-1)}{(x+1)/2 +2-k} $ is negative for all $2k-1\le x<x_0$ and positive for all $x>x_0$.) By Lemma~\ref{lem:ramsey},  $R(\omega(G),3)-1$ and $R(\omega(G),3)-2$ are the two largest allowed values of $\deg(v)$.
So by Corollary~\ref{cor:Q}, if $v$ is a vertex of a claw-free graph $G$ with $\deg(v) \ge 2\omega(G)-1$, then $\deg_{G^2}(v) \leq \max \{f(2 \omega(G)-1), f(R(\omega(G),3)-1), f(R(\omega(G),3)-2)\}$, yielding
\begin{align}\label{eqn:Q}
\deg_{G^2}(v) 
\le & \max\bigg\{2\omega(G)-1+(\omega(G)-1/2)(\omega(G)-1),\nonumber\\
&\,\,\,\,\,\,\,\,\,\,\,\,\,\,\,\,R(\omega(G),3)-2 + \frac{(R(\omega(G),3)-2)(\omega(G)-1)}{(R(\omega(G),3)-1)/2+2-\omega(G)},\\
&\,\,\,\,\,\,\,\,\,\,\,\,\,\,\,\,R(\omega(G),3)-1 + \frac{(R(\omega(G),3)-1)(\omega(G)-1)}{R(\omega(G),3)/2+2-\omega(G)}
\bigg\}.\nonumber
\end{align}
Moreover,~\eqref{eqn:Q} remains valid when we substitute $R(\omega(G),3)$ with any upper bound.
It is known~\cite{ErSz35} that $R(\omega(G),3)\le \binom{\omega(G)+1}{2}$. With this and some routine calculus,~\eqref{eqn:Q} implies that $\deg_{G^2}(v) \le 2\omega(G)(\omega(G)-1)$ provided $\omega(G)\ge 3$. Since those $v$ with $\deg(v) \le 2\omega(G)-2$ have $\deg_{G^2}(v) \le 2\omega(G)(\omega(G)-1)$ by Lemma~\ref{lem:2ndclique}, we have the following ``trivial'' bound on $\chi(G^2)$. This was proved not via $\Delta(G^2)$ but by a different method in~\cite{JKP16+}.

\begin{Corollary}\label{cor:trivial}
If $G$ is a claw-free graph, then $\chi(G^2) \le \Delta(G^2)+1\le 2\omega(G)(\omega(G)-1)+1$.
\end{Corollary}

Also~\eqref{eqn:Q} implies that, if $v$ is a vertex of a claw-free graph $G$ with $\deg(v) \ge 2\omega(G)-1$, then $\deg_{G^2}(v) \le \frac14(5\omega(G)^2 -2\omega(G)+1)-1$ provided $\omega(G)\ge 5$. We use this for the following.

\begin{Theorem}\label{thm:ge5}
Let $G = (V,E)$ be a connected claw-free graph with $\omega(G)=\omega \ge 5$. Then one of the following is true:
\begin{enumerate}
\item\label{itm:ge5,line}
$G$ is the line graph $L(F)$ of a graph $F$ of maximum degree $\omega$; or
\item\label{itm:ge5,deg}
there exists $v\in V$ with $\deg_{G^2}(v) \le 2\omega(\omega-1)-4$ such that $\deg_{G^2}(x) \le 2\omega(\omega-1)-3$ for all $x\in N_G(v)$.
\end{enumerate}
\end{Theorem}

\begin{proof}
By the last remark (which followed from Corollary~\ref{cor:Q}), for $v\in V$ with $\deg(v) \ge 2\omega-1$, we have that $\deg_{G^2}(v) \le \frac14(5\omega^2 -2\omega+1)-1\le2\omega(\omega-1)-4$ since $\omega\ge5$.

For $v\in V$ with $\deg(v) \le 2\omega-3$, we have by Lemma~\ref{lem:2ndclique} that $\deg_{G^2}(v) \le \omega(2\omega-3)\le2\omega(\omega-1)-4$ since $\omega\ge5$.

Let $v\in V$ with $\deg(v) = 2\omega-2$. If $G[N(v)]$ is not the disjoint union of two cliques, then $|Z(v)| \ge 2$.
(Clearly $|Z(v)| >0$ if $G[N(v)]$ is not the disjoint union of two cliques, but if on the contrary $|Z(v)|=1$ then let $w\in N(v)$ be the unique vertex such that there exist $x,y\in N(v)$ for which $xw,wy\in E$, $xy\notin E$. By the uniqueness of $w$, $x$ does not have any neighbours in $N(v)$ in common with $y$. Moreover, $(\{x\}\cup N(x)) \cap N(v)$ is a clique, because otherwise we would either have a claw or $x\in Z(v)$. By the uniqueness of $w$, $(\{x\}\cup N(x)) \cap N(v) \subseteq N(w) \cup \left\{w \right\}$. The same arguments hold with the roles of $x$ and $y$ exchanged. It follows that $G[N(v)]$ is the union of two cliques with exactly one vertex in common. Since each clique in $G[N(v)]$ is of size at most $\omega-1$, this is a contradiction to $\deg(v) = 2\omega-2$.)
It then follows by Lemma~\ref{lem:Z} that $\deg_{G^2}(v) \le (2\omega-1)(\omega-1) \le2\omega(\omega-1)-4$ since $\omega\ge5$.

We have shown that one of the following two possibilities must hold for $G$:
\begin{enumerate}
\item
for every $v\in V$ it holds that $G[N(v)]$ is the disjoint union of two cliques of size $\omega-1$ or that same graph with one extra edge between the two cliques, or the disjoint union of two cliques one of size $\omega-2$ the other of size $\omega-1$; or
\item
there is some $v\in V$ with $\deg_{G^2}(v) \le 2\omega(\omega-1)-4$.
\end{enumerate}
In the former situation, $G$ is the line graph of a graph of maximum degree $\omega$.

Let us call a vertex $v$ {\em very bad} if $\deg_{G^2}(v) \ge 2\omega(\omega-1)-2$. We already observed that $v$ must then have $\deg(v)=2\omega-2$. As argued just above, Lemma~\ref{lem:Z} implies that the neighbourhood of $v$ induces a disjoint union of two cliques of size $\omega-1$. Moreover, using Lemma~\ref{lem:2ndclique}, we have that for every neighbour $x$ of $v$ the neighbourhood of $x$ induces the disjoint union of two cliques of size $\omega-1$, or that same graph plus one more edge, or the disjoint union of two cliques one of size $\omega-2$ the other of size $\omega-1$.
This implies that, for every vertex $v$ for which we showed above that $\deg_{G^2}(v) \le 2\omega(\omega-1)-4$ (not including those cases corresponding to the promised line graph of maximum degree $\omega$), it also holds that $N(v)$ does not contain a very bad vertex. This completes the proof.
\end{proof}

\begin{proof}[Proof of Theorem~\ref{thm:col,ge5}]
Together with the trivial bound, Theorem~\ref{thm:ge5} certifies that we can apply Lemma~\ref{lem:greedy} with $K=K'=\max\{\chi'_s(\omega), 2\omega(\omega-1)-3\}-1$.
\end{proof}

We wanted to illustrate how our methods could extend to larger values of $\omega(G)$. It is likely that Theorem~\ref{thm:ge5} can be improved, particularly since we did not use the full strength of Lemma~\ref{lem:greedy}. On the other hand, since the Erd\H{o}s--Ne\v{s}et\v{r}il conjecture itself is open apart from the case of graphs of maximum degree at most $3$, we leave this to further investigation.

\subsection*{Acknowledgments}

We thank Luke Postle for alerting us to a subtlety in an earlier version. We also thank R\'emi de Joannis de Verclos and Lucas Pastor for helpful discussions in relation to Section~\ref{sec:greedy}.



\bibliographystyle{abbrv}
\bibliography{squarecfsmall}

\bigskip

\small
\bigskip
\noindent

\end{document}